\documentclass[11pt,letterpaper]{amsart}
\usepackage[leqno]{amsmath}
\usepackage{amssymb}
\usepackage{enumerate}
\usepackage{color}
\usepackage{subfigure}
\usepackage{graphicx}
\usepackage{hyperref}
\usepackage{mathtools}
\usepackage{mathrsfs}
\usepackage{mathabx}

\usepackage[english]{babel}
\usepackage{comment}
\usepackage[margin=0.3in]{geometry}
\newgeometry{includefoot,left=2.8cm,right=2.8cm, bottom=2.8cm,top=2.8cm}

\usepackage{microtype}
\usepackage{stmaryrd}

\numberwithin{equation}{section}

\newtheorem{theorem}{Theorem}[section]
\newtheorem{corollary}[theorem]{Corollary}
\newtheorem{lemma}[theorem]{Lemma}
\newtheorem{proposition}[theorem]{Proposition}

\theoremstyle{definition}
\newtheorem{definition}[theorem]{Definition}
\newtheorem{remark}[theorem]{Remark}

\newcommand{\E}{\mathbb{E}}
\newcommand{\p}{\mathbb{P}}
\newcommand{\N}{\mathbb{N}}
\newcommand{\Z}{\mathbb{Z}}

\newcommand{\var}{\mathrm{Var}}
\newcommand{\R}{\mathbb{R}}
\newcommand{\dist}{\mathrm{dist}}
\newcommand{\F}{\mathscr{F}}
\newcommand{\Ff}{\mathfrak{f}}
\newcommand{\FD}{\mathfrak{D}}
\newcommand{\CA}{\mathcal{A}}
\newcommand{\CB}{\mathcal{B}}
\newcommand{\CD}{\mathcal{D}}
\newcommand{\CE}{\mathcal{E}}
\newcommand{\CG}{\mathcal{G}}

\newcommand{\CL}{\mathcal{L}}
\newcommand{\CN}{\mathcal{N}}
\newcommand{\CR}{\mathcal{R}}
\newcommand{\CS}{\mathcal{S}}
\newcommand{\CT}{\mathcal{T}}
\newcommand{\CV}{\mathcal{V}}
\newcommand{\SN}{\mathscr{N}}
\newcommand{\Exp}{\mathrm{Exp}}
\newcommand{\Area}{\mathrm{Area}}

\newcommand{\Unif}{\mathrm{U}}
\newcommand{\rad}{\mathrm{rad}}
\newcommand{\Fs}{\mathrm{Fs}}
\newcommand{\Po}{\mathrm{Po}}
\newcommand{\BBM}{\mathrm{BBM}}

\newcommand{\half}{\mathrm{half}}
\newcommand{\origin}{\mathbf{0}}

\renewcommand{\epsilon}{\varepsilon}

\newcommand{\ol}{\overline}

\newcommand{\wt}{\widetilde}

\begin{document}

\title{Scaling properties of a random Yule tree embedded in $\R^\FD$}

\author{Lukas Schoug}
\address{Department of Mathematics and Statistics, University of Helsinki, Finland}
\email{lukas.schoug@helsinki.fi}


\begin{abstract}
We study a random tree, which was introduced in \cite{antz2015network} as part of a model of a neuronal network.  Realising a scaling relation for the law of the tree, we can use elementary techniques to derive asymptotic results on the geometry as time goes to infinity.

\smallskip
\noindent \textbf{Keywords:} Scaling, birth process, random growth model.

\noindent \textbf{MSC2020 subject classifications:} Primary 60J80, Secondary 60G50.
\end{abstract}

\maketitle

\parindent 0 pt
\setlength{\parskip}{0.20cm plus1mm minus1mm}

\section{Introduction}
We study the growth of a random tree, $(\CT_\lambda(t))_{t \geq 0}$, in $\R^\FD$, $\FD \in \N^+ = \{1,2,\dots\}$, starting from the origin and growing at unit speed in a random direction, branching after a random exponentially distributed time. Each new interval formed by branching evolves independently of the others as the first interval and branches in the same manner. (This tree can be realised in the same way as a branching Brownian motion, where the branching rate for a particle is $\lambda$, the Brownian trajectories replaced by lines growing at speed $1$ and drawn in a uniform direction, and where one keeps track of the entire history of the process.) In particular, this is an embedding of a Yule process with rate $\lambda$ (that is, a birth process with individual birth rate $\lambda$) into $\R^\FD$ as a tree, such that each time interval between births is embedded as a line in $\R^\FD$, attached to the endpoint of the line corresponding to the previous birth of that particular branch, and growing in a uniformly randomly chosen direction. As such, we call $\CT_\lambda(t)$ a Yule tree.

This model was introduced in two dimensions in \cite{antz2015network} to model a neuronal network and was studied further in \cite{acdt2019trees}, \cite{gt2019decay} and \cite{goriachkin2023thesis} (the latter paper also considering a three-dimensional version). In those papers each such random tree was supposed to simulate a neuron and the authors studied networks consisting of many such trees (we will not touch upon the biology background --- for more information on that, see \cite{antz2015network} and the references therein) and the growth of the trees define a directed graph. The model is studied both analytically and statistically: in \cite{antz2015network} the moment generating function of the total length of the network, as well as a functional equation for connection probabilities are derived and simulations predicting features of the model, such as density, length distribution etc.\ are carried out. In \cite{acdt2019trees}, connectivity properties related to the graph structure of the network were studied, deriving expectations of in- and out-degrees, and numerical experiments were performed, estimating frequencies of connection, shortest paths, in- and out-degrees and clustering coefficients. Next, in \cite{gt2019decay}, the authors study the decay of the connection probabilities, by numerically solving the functional equations describing the connection probabilities obtained in \cite{antz2015network} and the corresponding equations in three dimensions (which they derive in the paper).  Finally, in Paper~IV of \cite{goriachkin2023thesis}, the author studies a the distance between branching points and the origin, as well as a version where the branches of the tree stop growing randomly at a prescribed rate.

Extensive work has also been put into simulating neuronal networks, providing statistical properties of the models. A recent simulation environment is NETMORPH, developed in \cite{koeneetal2009netmorph} and studied and evaluated further in \cite{ammhtl2011netmorph} and \cite{mm2013thesis}. The tree $\CT_\lambda$ can be simulated in that environment as well\footnote{The simulations of $\CT_\lambda$ in this paper, however, are not done using NETMORPH.}.

However, interesting as this process is, it is not yet well understood. While simulations and statistical analyses have been carried out, relatively few results about the model have been rigorously proved. Natural questions, such as the radius or shape turn out to be cumbersome much due to the fact that each branching penalises the growth of a branch (but can still help the ``average'' growth of the tree, as in the growth in all directions simultaneously, since there will be more branches growing). Indeed, the behaviour of a branch depends heavily on how many times it branches and increasing the branching intensity to infinity actually makes an entire branch (grown in a finite time) converge to $\origin = \{ (0,\dots,0) \} \subset \R^\FD$ in the topology determined by the Hausdorff distance (see Lemma~\ref{lem:branch_to_zero}). Moreover, the heavy dependence of a branch on the number times it branches makes it difficult to use many-to-one and many-to-few formulas, which are powerful tools in the theory of branching processes used to deduce results on the whole branching structure from the law of an individual branch (see e.g.\ \cite{hr2017manytofew} and \cite{hhkp2022manytofew}).

In this paper, we study the geometry of $\CT_\lambda$ from a completely analytical perspective. The main features of interest in this paper are the radius and shape of the tree, for fixed intensity $\lambda > 0$ and large times $t>0$ and for fixed time and high intensities. As we shall see, there is a certain correspondence between these regimes and the parts of Theorems~\ref{thm:main_result} and~\ref{thm:main_result_1} are deduced from the structure of $\CT_\lambda(1)$ in the high intensity regime, by virtue of scaling (Lemma~\ref{lem:scaling}). We now state our main results. Let $\CR_\lambda(t) = \sup_{x \in \CT_\lambda(t)} |x|$ be the radius of the tree $\CT_\lambda$ at time $t$. Our first main result concerns $\CR_\lambda$.
\begin{theorem}\label{thm:main_result}
Fix $\alpha > 0$. For any $\epsilon > 0$, there exists $\lambda_0 = \lambda_0(\alpha,\epsilon) > 0$ such that for all $\lambda \geq \lambda_0$, we have
\begin{align*}
	\p(\CR_\lambda(1) \leq 1/2 - \epsilon) \leq (1+\alpha) e^{-\epsilon \lambda}.
\end{align*}
Moreover, for any $\lambda > 0$, almost surely, $\CR_\lambda^{\mathrm{LI}} \coloneqq \liminf_{t \to \infty} \CR_\lambda(t)/t \geq 1/2$ and the law of $\CR_\lambda^{\mathrm{LI}}$ does not depend on $\lambda$. 
\end{theorem}

For an explicit bound on $\p(\CR_\lambda(1) \leq 1/2 - \epsilon)$ for all $\lambda > 0$ and $\epsilon > 0$, see the end of the proof of Proposition~\ref{prop:radius_bound}.

\begin{remark}
We remark that the model exhibits a $0$-$1$-type behaviour (see Lemma~\ref{lem:zero_one_law_radius}). In particular, if $\p(\CR_\lambda(1) \geq d) \geq p$ holds for some $d > 1/2$ and all $\lambda$ large enough, then $\CR_\lambda^{\mathrm{LI}} \geq d$ almost surely (see Corollary~\ref{cor:liminf_lower_bound}) and the first assertion of Theorem~\ref{thm:main_result} holds with $d$ in place of $1/2$ with a slightly weaker exponent.
\end{remark}

Before stating the next results, we recall some definitions. For $x \in \R^\FD$ and $r>0$, we let $B(x,r) = \{y \in \R^\FD: |x-y| < r \}$ denote the open ball in $\R^\FD$ with centre $x$ and radius $r$ and for a set $A \subset \R^\FD$, we let $B(A,r) = \{ y \in \R^\FD: \dist(A,y) < r \}$, where $\dist$ is the Euclidean distance, that is, $B(A,r)$ is the open $r$-neighbourhood of $A$.  The \emph{Hausdorff distance} is the metric $\dist_H$ on subsets of $\R^\FD$ given by 
\begin{align*}
	\dist_H(A,A') = \inf\{ \epsilon > 0: A \subset B(A',\epsilon) \ \text{and} \ A' \subset B(A,\epsilon) \},
\end{align*}
for $A,A' \subset \R^\FD$. Thus, we say that a sequence $(A_n)$ of random sets converge to a set $A$ in probability with respect to the topology determined by the Hausdorff distance if $\p( \dist_H(A_n,A) > \epsilon) \to 0$ as $n \to \infty$ for each $\epsilon > 0$.

\begin{theorem}\label{thm:main_result_1}
$\CT_\lambda(1) \cap B(0,1/2)$ converges to $B(0,1/2)$ in probability with respect to the topology determined by the Hausdorff distance, as $\lambda \to \infty$. Moreover, for any $\delta > 0$, $C> 2$ and $a \in (0,1)$, there exists $t_* = t_*(\lambda,C,a,\delta) > 0$, such that
\begin{align*}
	\p( \exists x \in B(0,t/2):\dist(\CT_\lambda(t),x) > \delta t) \leq C \delta^{-2\FD} e^{- \tfrac{(1-a)}{16} \lambda \delta t},
\end{align*}
for all $t \geq t_*$.
\end{theorem}

We expect there to be some $d \in [1/2,1]$ (which might depend on the dimension $\FD$) such that $\CT_\lambda(1) \to B(0,d)$ in probability with respect to the topology determined by the Hausdorff distance, as $\lambda \to \infty$. See the discussion in the beginning of Section~\ref{sec:connection_infinity}. Furthermore, we expect this number $d$ to be the almost sure limit of $\CR_\lambda(t)/t$ as $t \to \infty$. 

The outline of the paper is as follows. We begin in Section~\ref{sec:model} by introducing the model, revisiting previous results, and discussing seemingly similar models. In Section~\ref{sec:branch} we prove a convergence result for the branches of the model. Then, in Section~\ref{sec:geometry}, we begin by proving the scaling rule for $\CT_\lambda$, before moving on to the radius bounds in Section~\ref{sec:radius} and the convergence results in Section~\ref{sec:connection}. Lastly, in Appendix~\ref{app:asymptotics} we prove the asymptotics of first success and exponential distributions used in Section~\ref{sec:radius}.

\subsection*{Notation and conventions}
We denote by $\R$ the real numbers and $\N$ the set of non-negative integers, $\N^+$ the positive integers and $\CS^n = \{ x \in \R^{n+1}: | x | = 1\}$ the unit $n$-sphere. For a set $D \subset \R^\FD$ and a positive number $a > 0$, we write $D/a = \{x \in \R^\FD: a x \in D \}$. Furthermore, for $x \in \R^\FD$ and $r > 0$, we let $B(x,r)$ denote the open ball of radius $r$. We let $\origin = \{(0,\dots,0)\} \subset \R^\FD$.

We denote by $\Exp(\lambda)$ the exponential distribution parametrised by rate, that is, such that if $Y \sim \Exp(\lambda)$, then $\E[Y] = 1/\lambda$.

\subsection*{Acknowledgements}
The author was supported by the Finnish Academy Centre of Excellence FiRST and the ERC starting grant 804166 (SPRS). We thank Tatyana Turova and Vasilii Goriachkin for interesting discussions and we are grateful to Turova for introducing us to the model and an anonymous referee for helpful comments on the paper. Finally, we express our gratitude to Frida Fejne for simulating the random tree.

\section{The model $\CT_\lambda(t)$: a random tree in Euclidean space}\label{sec:model}
Fix $\FD \in \N^+$ and for $x \in \R^\FD$, let $\CT_\lambda^x(t) = \CT_\lambda^{x,\FD}(t)$ be a random set defined as follows.
Let $\CT_\lambda^x(0) = x$. Draw a vector $\nu \in \CS^{\FD-1}$ uniformly at random and an exponential random variable $\tau$ with expectation $1/\lambda$, independently of $\nu$, and for $t < \tau$ let $\CT_\lambda^x(t) = \{x + s\nu\}_{s \in [0,t]}$, i.e., for $t < \tau$ it is an interval of length $t$ drawn in the direction $\nu$. The end of the interval is called an active end, or a leaf. At time $\tau$, two new, independent, intervals start to grow from the leaf in two new, independently and uniformly randomly drawn directions, and what was the leaf is now no longer growing, hence is for time $t > \tau$ referred to as a \emph{branching point}. Both of the new intervals branch independently in the same manner as the first interval, with intensity $\lambda$. We write $\CT_\lambda(t) = \CT_\lambda^0(t)$. We denote by $\CS_\lambda(t)$ the set of branching points of $\CT_\lambda(t)$.

While, in the case $\FD \leq 2$, this process is self-intersecting and hence does form loops, we still refer to it as a tree. This is because if we define the graph $\CG_\lambda(t) = (\CV_\lambda(t),\CE_\lambda(t))$, to be such that $\CV_\lambda(t)$ is the union of $\CS_\lambda(t)$ and the leaves of $\CT_\lambda(t)$ and $\{v_1,v_2\} \in \CE_\lambda(t)$ if $v_1,v_2 \in \CV_\lambda(t)$ and there is an interval in $\CT_\lambda(t)$ connecting $v_1$ and $v_2$, then $\CG_\lambda(t)$ is a tree.

Every path $\gamma:[0,t] \to \CT_\lambda(t)$, starting at $\CT_\lambda(0)$ and ending at some leaf of $\CT_\lambda(t)$, travelling at unit speed and leaving an interval in $\CT_\lambda(t)$ only at the point where it branches into two new intervals, is called a \emph{branch} of $\CT_\lambda(t)$. Note here that almost surely, no two branching points coincide, so for each time $t > 0$ and each leaf $v$ of $\CT_\lambda(t)$ there is a unique branch of $\CT_\lambda(t)$ starting at $0$ and ending at $v$. Moreover, each branch of $\CT_\lambda(t)$ has length $t$. In (the equivalent) Definitions~\ref{def:branch1} and~\ref{def:branch2} we define a process $(\CB_\lambda(t))_{t \geq 0}$ which has the law of a branch of $\CT_\lambda(t)$ and we henceforth denote by $\CB_\lambda$ a branch of $\CT_\lambda$. Moreover, we denote by $\CB_\lambda([s_1,s_2])$ the set of points traced out by $\CB_\lambda$ in the time interval $[s_1,s_2]$.

We say that a point $y \in \CT_\lambda(t)$ is a \emph{descendant} of a point $x \in \CT_\lambda(t)$ if there is a branch $\CB_\lambda$ of $\CT_\lambda(t)$ and times $t_1<t_2 \leq t$ such that $x = \CB_\lambda(t_1)$ and $y = \CB_\lambda(t_2)$. We say that $x \in \CT_\lambda(t)$ is an \emph{ancestor} of $y \in \CT_\lambda(t)$ if $y$ is a descendant of $x$. The \emph{descending tree} at time $t$ of a point $x \in \CT_\lambda(t)$ is the set of descendants $y \in \CT_\lambda(t)$ of $x$. A \emph{double point} $\CT_\lambda(t)$ is a point $x \in \CT_\lambda(t)$ for which there are two branches $\CB_\lambda^1$, $\CB_\lambda^2$ of $\CT_\lambda(t)$ and times $0 \leq t_1,t_2 \leq t$ such that $x = \CB_\lambda^1(t_1) = \CB_\lambda^2(t_2)$ and $\CB_\lambda^1([0,t_1]) \neq \CB_\lambda^2([0,t_2])$. We let $\CD_\lambda(t)$ denote the set of double points of $\CT_\lambda(t)$. Note that this definition allows for $\CB_\lambda^1$ and $\CB_\lambda^2$ to be the same, provided that $x$ is hit by the branch at two different points in time. If this is the case, we say that $x$ is \emph{hit twice by the same branch}. When we refer to a point $x \in \CD_\lambda(t)$ \emph{hit by two different branches}, we mean a double point which is not hit twice by the same branch (even though most points of $\CT_\lambda(t)$ will actually be hit by more than one branch, in the sense that there are several branches that hit the same points before they branch). Note that the set of triple points of $\CT_\lambda(t)$ (that is, the set of points $x \in \CT_\lambda(t)$ such that there exist branches $\CB_\lambda^1$, $\CB_\lambda^2$ and $\CB_\lambda^3$ of $\CT_\lambda$ and times $0 \leq t_1, t_2, t_3 \leq t$, such that $x = \CB_\lambda^1(t_1) = \CB_\lambda^2(t_2) = \CB_\lambda^3(t_3)$ and such that $\CB_\lambda^i([0,t_i]) \neq \CB_\lambda^j([0,t_j])$ whenever $i \neq j$) is almost surely empty. Moreover, $\CS_\lambda(t) \cap \CD_\lambda(t) = \emptyset$ almost surely. Note that there are four different cases for the law of the descending tree of a point $x \in \CT_\lambda(t)$.
\begin{enumerate}[(1)]
	\item\label{it:dt1} The descending tree of a point $x \in \CS_\lambda(t)$ hit at time $s$ has the law of the union of two independently sampled trees with law $\CT_\lambda^x(t-s)$.  We refer to those two trees as the descending trees of $x$ and it will be clear to the reader what is meant.
	\item\label{it:dt2} The descending tree of a point $x \in \CT_\lambda(t) \setminus (\CD_\lambda(t) \cup \CS_\lambda(t))$ hit at time $s$ has the law of $\CT_\lambda^x(t-s)$, conditioned to start growing in the direction $\nu$ in which the branch hitting $x$ was growing at time $s$.
	\item\label{it:dt3} A point $x \in \CD_\lambda(t)$ that is hit by two different branches at times $s_1$ and $s_2$ have two descending trees, with the laws of $\CT_\lambda^x(t-s_1)$ and $\CT_\lambda^x(t-s_2)$, conditioned to start growing in the directions that their respective branches grew when hitting $x$.
	\item\label{it:dt4} A point $x \in \CD_\lambda(t)$ that is hit twice by the same branch, at times $s_1<s_2$, has two descending trees, with the laws of $\CT_\lambda^x(t-s_1)$ and $\CT_\lambda^x(t-s_2)$ conditioned to start growing in the directions that the branch hitting $x$ grew at times $s_1$ and $s_2$, respectively, and the latter is a subset of the former.
\end{enumerate}
\begin{remark}\label{rmk:independence_descending_trees}
We emphasise here that in case~\eqref{it:dt1}, the two descending trees are conditionally independent of $\CT_\lambda(s-)$ given $x$ and similarly in the cases~\eqref{it:dt2}--\eqref{it:dt4} the descending trees are conditionally independent the past given $x$ and their initial angles of growth. This independence is key in deriving bounds for the radius of the tree as well as functional equations for connection probabilities and expectations of densities, see Section~\ref{sec:connection}.
\end{remark}

\begin{figure}[ht!]
\centering
\includegraphics[width=120mm]{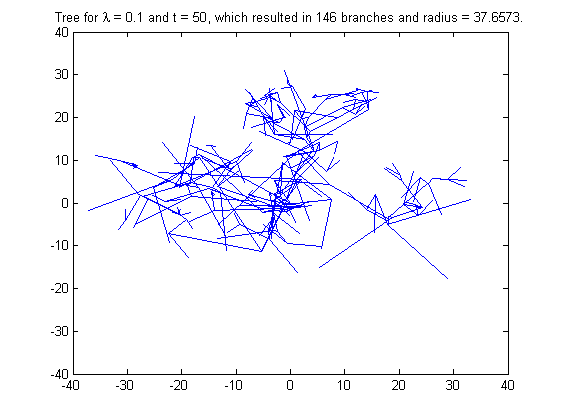}
\caption{Simulation of $\CT_\lambda(t)$ for $\FD = 2$ with $t = 50$ and $\lambda = 0.1$. \label{fig:t50l01}}
\end{figure}

\subsection{Previous results}
In this subsection we recall the previous mathematical results on the process $\CT_\lambda$. Let $\CN_\lambda(t)$ be the number of leaves of $\CT_\lambda(t)$ for $t \geq 0$ (with the understanding that $\CN_\lambda(0) = 1$).  We note that $\CN_\lambda(t)$ is a Yule process with rate $\lambda$, that is, a birth process with individual birth rate $\lambda$. Thus, the following holds, see Exercise~2.5.1 of \cite{norris1997book} or Proposition~6 of \cite{schoug2015thesis}. 
\begin{proposition}\label{prop:number_of_leaves}
$\CN_\lambda(t) \sim \Fs(e^{-\lambda t})$, i.e.,
\begin{align*}
\p( \CN_\lambda(t) = k ) = e^{-\lambda t} (1 - e^{-\lambda t})^{k-1},
\end{align*}
for $k=1,2,\dots$
\end{proposition}
Thus, $\E[\CN_\lambda(t)] = e^{\lambda t}$ and $\CN_\lambda(t)$ typically grows exponentially. This is a key property which drastically changes several geometric features, compared to if the growth had been slower. Indeed, this makes sure that very rare events for a branch actually occurs all the time for the tree, changing, for example, the radius of the tree significantly, as we will see in Section~\ref{sec:radius}.

Denote the total length of every interval in $\CT_\lambda(t)$ by $\CL_\lambda(t)$. The following was proved in Section~A.1 of \cite{antz2015network}.
\begin{lemma}\label{lem:length}
The moment generating function $\psi_{\CL_\lambda(t)}$ of $\CL_\lambda(t)$ is given by
\begin{align*}
\psi_{\CL_\lambda(t)}(x) = \frac{x-\lambda}{x e^{(\lambda - x)t}-\lambda}.
\end{align*}
In particular,
\begin{align*}
\E[ \CL_\lambda(t)] = \frac{e^{\lambda t} - 1}{\lambda}.
\end{align*}
\end{lemma}
Denote by $\CA_r^\lambda(t) = \Area(\{x \in \R^2: \dist(x,\CT_\lambda(t)) \leq r \})$. It is easy to see that
\begin{align*}
	\pi r^2 \leq \CA_r^\lambda(t) \leq 2r \CL_\lambda(t) + \frac{\pi r^2}{2} ( \CN_\lambda(t) + 1).
\end{align*}
We will see that the probability that $\pi / 4 \leq \CA_r^\lambda(1) \leq \pi(1+r)^2$ converges exponentially fast to $1$ as $\lambda \to \infty$, for any value of $r > 0$. Also, a functional equation for the function $q_r^\lambda(t,x) = \p( \dist(\CT_\lambda(t),x) > r)$ was derived in Section~3.2 of \cite{antz2015network}, but this will be stated in Section~\ref{sec:connection}.

\subsection{Differences in the model compared to previous work}
While the model is mostly the same, there are some key differences compared to the works \cite{antz2015network}, \cite{acdt2019trees} and \cite{gt2019decay}. One obvious restriction in this note is that we consider the growth of only one tree, whereas the main goal of the papers \cite{antz2015network}, \cite{acdt2019trees} and \cite{gt2019decay} is to study the network composed of many trees. A restriction that is in a way imposed in said papers is that the network of trees lives in a bounded region and growth of branches is killed off upon reaching the boundary of said region. We stress, however, that some properties are also derived without considering a bounded region and when only considering one tree. Finally, we mention that in \cite{goriachkin2023thesis}, the author studies the model when leaves stop growing at random.

Another difficulty that arises in the perspective of the above mentioned papers is the fitting of parameters to approximate the properties of actual neuronal networks. This is an aspect that we will not consider in this paper.

\subsection{Comparison to potentially similar branching processes}
There are other branching processes which might seem similar at first glance. Branching Brownian motion is naturally the one of first that come to mind and it has been studied extensively, see e.g.\ \cite{bramson1978displacement}, \cite{roberts2015bbmconsistent}, \cite{mallein2015displacement} and \cite{abbs2013bbmtip}. However, it satisfies very different large scale behaviour. Both of the processes typically have a radius of order $t$ (recall Theorem~\ref{thm:main_result} and see Theorem~1.1 of \cite{mallein2015displacement}), however, deducing a scaling rule for branching Brownian motion shows that this is not the case. Indeed, if $(\BBM_\lambda(t))_{t \geq 0}$ denotes a branching Brownian motion, where the branching rate is $\lambda$ (so that ordinary branching Brownian motion is $\BBM_1$), then we have by Brownian scaling and scaling of exponential random variables, that $\sqrt{\lambda} \BBM_\lambda(t)$ has the same law as $\BBM_1(\lambda t)$. Then, by Theorem~1.1 of \cite{mallein2015displacement}, the radius of $\BBM_\lambda(t)$ has order $\sqrt{2 \lambda} t$ (plus a logarithmic term), so that increasing $\lambda$ pushes the maximal displacement of branching Brownian motion further from the origin. In the model $\CT_\lambda(t)$, the same does not hold true, as $\CT_\lambda(t) \subset \ol{B(0,t)}$ for any $\lambda \geq 0$.

One might also consider branching random walks, for which one has good control of the speed of growth (or in fact \emph{speeds} of growth, see e.g.\ \cite{gantert2000branching}). However, typically when studying branching random walks, one is concerned with what happens at a certain step, whereas in the present model, the number of steps taken at a finite time $t > 0$ depends on the length of the steps, adding extra difficulty which to our best knowledge is not considered before in this context. 

For more on branching random walks and branching Brownian motions, see \cite{shi2012book} and \cite{zeitouni2016notes}.

\section{Branches and their approximate projections}\label{sec:branch}
We now give two equivalent definitions of a stochastic process which describes the law of a branch of $\CT_\lambda(t)$, $t \geq 0$. We denote by $\Unif(\CS^{\FD-1})$ the uniform distribution on $\CS^{\FD-1}$. In what follows, we consider a branch $(\CB_\lambda(t))_{t \geq 0}$ to be a stochastic process (in $\R^\FD$) and write $\CB_\lambda([s,t])$ to be the set of points traced out by $\CB_\lambda$ in the time interval $[s,t]$ for $0 \leq s < t$. 
\begin{definition}\label{def:branch1}
Let $t_i \in \Exp(\lambda)$ and $\nu_i \in \Unif(\CS^{\FD-1})$ for $i \in \N$. We let $(\CB_\lambda(t))_{t \geq 0}$ be the stochastic process
\begin{align*}
&\CB_\lambda(t) := (\CB_\lambda^1(t),\dots,\CB_\lambda^\FD(t)) = \sum_{i=0}^{N_{\CB}(t)} t_i \nu_i + \left(t-\sum_{i=0}^{N_{\CB}(t)} t_i \right) \nu_{N_\CB(t)+1},
\end{align*}
where $N_\CB(t) := \min \{ j \geq 0: \sum_{i=0}^j t_i \leq t < \sum_{i=0}^{j+1} t_i \}$.
\end{definition}
The second definition is the following.
\begin{definition}\label{def:branch2}
Let $N_\CB(t)$ be a Poisson process with intensity $\lambda$. Define the stochastic process $(\CB_\lambda(t))_{t \geq 0}$ as
\begin{align*}
\CB_\lambda(t) := (\CB_\lambda^1(t),\dots,\CB_\lambda^\FD(t)) = \sum_{i=0}^{N_\CB(t)} t_i \nu_i,
\end{align*}
where, conditionally on $N_\CB(t)=n$, $t_0$ is the time of the first jump of $N_\CB(t)$ and $t_i$ is the time between the $i$th and the $i+1$st jump of $N_\CB(t)$, for $i \in \{1,...,n-1\}$ and $t_n = t - \sum_{i=0}^{n-1} t_i$.
\end{definition}
\begin{figure}[h!]
	\centering
		\includegraphics[width=0.49\textwidth]{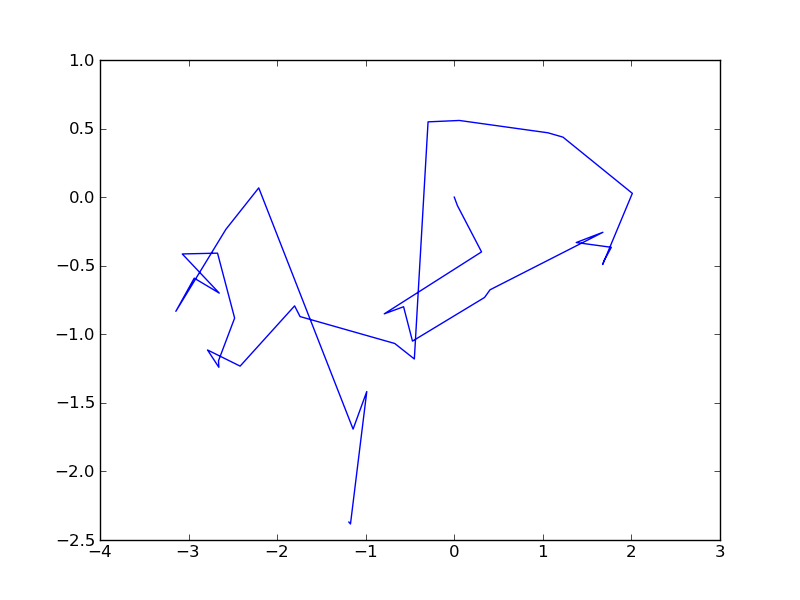} \hspace{0.01\textwidth}\includegraphics[width=0.49\textwidth]{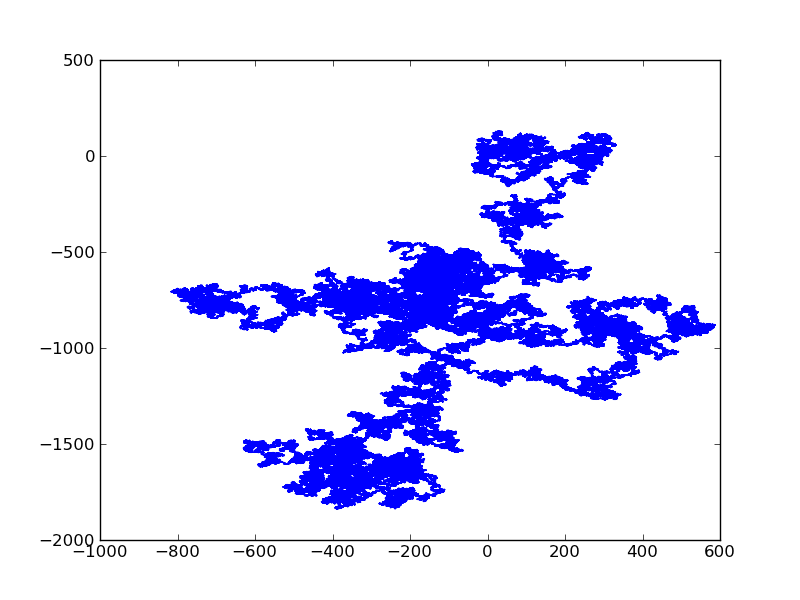}
		\caption{{\bf Left:}  Simulation of $\CB_\lambda(t)$ for $\FD = 2$, with $t = 20$ and $\lambda = 2$. {\bf Right:} Simulation of $\CB_\lambda(t)$ for $\FD = 2$, with $t = 2 \times 10^6$ and $\lambda = 2$.}
		\label{fig:B_t20_l2}
\end{figure}

However, studying the sums $\sum_{i=0}^{N_\CB(t)} t_i \nu_i$ in the above definitions is not straightforward at all, as the number of summands depends on the magnitude of the $t_i$'s. In studying a branch of $\CT_\lambda(t)$, we consider the simpler process, $(X_k)_{k \in \N} = (X_k^1,\dots,X_k^\FD)_{k \in \N}$, defined by
\begin{align*}
X_k^n \coloneqq \sum_{i=0}^k t_i \nu_i^n,
\end{align*}
where $(t_i)$ and $(\nu_i)$ are independent sequences such that $t_i \in \Exp(\lambda)$ and $\nu_i = (\nu_i^1,\dots,\nu_i^\FD) \sim \Unif(\CS^{\FD-1})$ for all integers $i \geq 1$ and the $t_i$ (resp.\ $\nu_i$) is independent of $t_j$ (resp.\ $\nu_j$) whenever $j \neq i$. Let $\F_k = \sigma( t_j, \nu_j: 1 \leq j\leq k)$. For each $n =1,\dots,\FD$, the process $(X_k^n)_{k \in \N}$ will behave approximately as the projection of a branch onto the $n$th coordinate axis. We will begin by stating some basic properties which are straightforward to check.
\begin{lemma}\label{lem:projections}
We have that
\begin{enumerate}[(i)]
  \item $\E[X_k^n] = 0$,
  \item $\var [X_k^n] \leq \frac{2(k+1)}{\lambda^2}$,
  \item $X_k^m$ and $X_k^n$ are uncorrelated for $m \neq n$,
  \item $X_k^n$ is a martingale with respect to the filtration $(\F_k)_{k \geq 1}$.
\end{enumerate}
\end{lemma}

Suppose that $X_k = (X_k^1,\dots,X_k^\FD)$, then we would like to know the radius of $X_k$, i.e.,
\begin{align*}
R_X(k) = \max_{j \leq k} |X_j| = \max_{j \leq k} \left( \sum_{n=1}^\FD (X_j^n)^2 \right)^{1/2}.
\end{align*}
We note that if $R_X^n(k) = \max_{j \leq k} |X_j^n|$, then
\begin{align}\label{eq:radius_bound_approximate_branch}
\max_{1 \leq n \leq \FD} R_X^n(k) \leq R_X(k) \leq \sum_{n = 1}^\FD R_X^n(k).
\end{align}
The above observations will be used to describe the behaviour of a branch $\CB_\lambda$ when $\lambda \to \infty$.
\begin{lemma}\label{lem:branch_to_zero}
For each $t >0$, the range of the branch $\CB_\lambda([0,t])$ converges to $\origin$ in probability with respect to the topology determined by the Hausdorff distance as $\lambda \to \infty$. That is, the random variable $\max_{0 \leq s \leq t} |\CB_\lambda(s)|$ converges to $0$ in probability as $\lambda \to \infty$.
\end{lemma}
\begin{proof}
We begin by noting that we can couple $\CB_\lambda$ with $(X_j)_{j \in \N}$ so that if $(t_j)_{j \in \N}$ are the times between the changes of direction in the branch $\CB_\lambda$ and $t_k^* = \sum_{j \leq k} t_j$ are the times at which the changes are made, then $\CB_\lambda( t_k^*) = X_k$, and such that the direction that the branch grows in a given time interval is the same as that of $X$ in the corresponding one. That is, the processes $\CB_\lambda$ and $(X_j)_{j \in \N}$ change direction at the same times and the growth directions are the same.  We note that for each $k \in \N$,
\begin{align}
	\max_{0 \leq s \leq t_k^*} |\CB_\lambda(s)| = \max_{j \leq k} |\CB_\lambda(t_j^*)| = R_X(k).
\end{align}
For $n = 1,\dots,\FD$, we have by the Kolmogorov martingale inequality and Lemma~\ref{lem:projections}, that
\begin{align}
	\p( R_X^n(k) \geq \epsilon) \leq \max_{j \leq k} \frac{\var(X_k^n) }{\epsilon} \leq \frac{2(k+1)}{\epsilon \lambda^2}.
\end{align}
Consequently, by~\eqref{eq:radius_bound_approximate_branch},
\begin{align}\label{eq:branch_bound_1}
	\p( R_X(k) \geq \epsilon) \leq \sum_{n=1}^\FD \p( R_X^n(k) \geq \epsilon/\FD) \leq \frac{2\FD^2(k+1)}{\epsilon \lambda^2}.
\end{align}
Let $N_\CB(t)$ be as in Definition~\ref{def:branch1} and recall that $N_\CB(t) \sim \Po(\lambda t)$. Thus, we have that
\begin{align}\label{eq:branch_bound_2}
	\p\!\left( \max_{0 \leq s \leq t} | \CB_\lambda(s)| \geq \epsilon \right) \leq \p( R_X(k) \geq \epsilon) + \p( N_\CB(t) > k).
\end{align}
Next we note that for any integer $k > \lambda t$, we have by a Chernoff bound that
\begin{align}\label{eq:poisson_bound}
	\p( N_\CB(t) \geq k) \leq \frac{(e \lambda t)^k e^{-\lambda t}}{k^k},
\end{align}
and hence, choosing $k = \lceil (\lambda t)^{3/2} \rceil$, we have by~\eqref{eq:branch_bound_1},~\eqref{eq:branch_bound_2} and~\eqref{eq:poisson_bound} that
\begin{align*}
	\p\!\left( \max_{0 \leq s \leq t} | \CB_\lambda(s)| \geq \epsilon \right) = O(\lambda^{-1/2}).
\end{align*}
Thus $\max_{0 \leq s \leq t} |\CB_\lambda(s)| \to 0$ in probability as $\lambda \to \infty$ and hence the proof is done.
\end{proof}
For more on the processes $X$ and $\CB_\lambda$ in the case $\FD = 2$, see \cite{schoug2015thesis}.

\section{On the geometry of $\CT_\lambda(t)$}\label{sec:geometry}
In this section, we study the geometry of the tree $\CT_\lambda(t)$. We begin by noting the following scaling relation, which is the key observation for the proofs of the present paper.
\begin{lemma}\label{lem:scaling}
Fix $s > 0$. Then, $\CT_\lambda(t)/s$ and $\CT_{\lambda s}(t/s)$ have the same law. In particular, the three sets $\CT_\lambda(t)/t$, $\CT_{\lambda t}(1)$ and $\CT_1(\lambda t)/\lambda$ have the same law.
\end{lemma}
\begin{proof}
This follows since if $\tau \sim \Exp(\lambda)$ and $s > 0$, then $\tau/s \sim \Exp(\lambda s)$. In other words, rescaling the lengths of the intervals by dividing by a number $s > 0$,  corresponds to increasing the splitting intensity by a factor of $s$ and dividing the growth time of the tree by $s$ (since the length of any branch is then divided by $s$).
\end{proof}
Thus, by Lemma~\ref{lem:scaling}, in studying $\CT_\lambda(t)$ for large $t$, we can (for some properties) instead study $\CT_\lambda(1)$ for large $\lambda$. Note here that any behaviour that is of scale $o(t)$, such as lower order radius terms or possibly fluctuations of the outer boundary, will not show up when rescaling as above.

\begin{figure}[h!]
	\centering
		\includegraphics[width=0.49\textwidth]{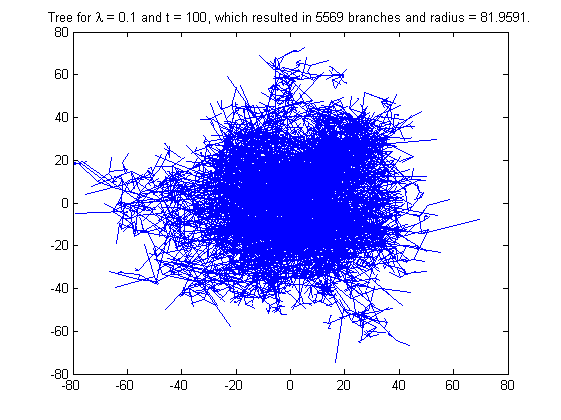} \hspace{0.01\textwidth}\includegraphics[width=0.49\textwidth]{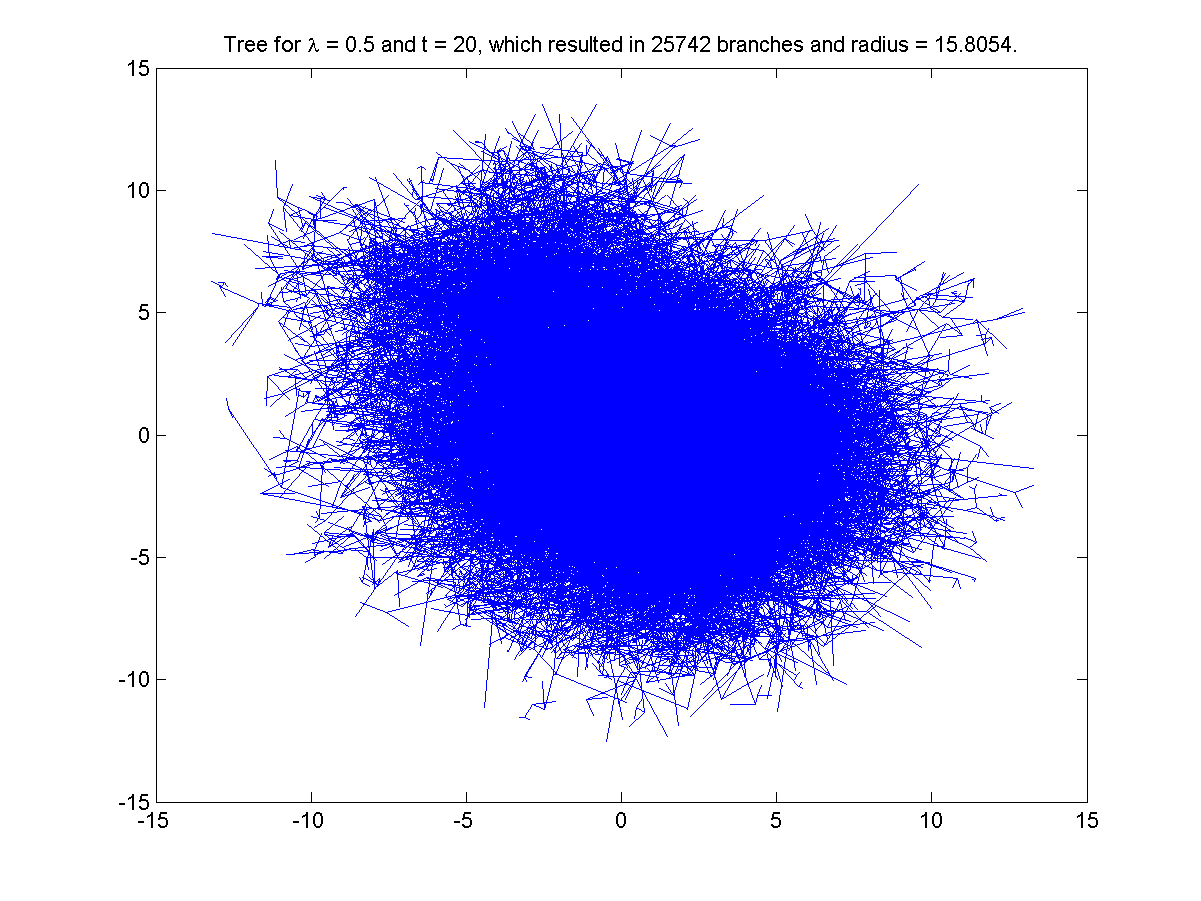}
		\caption{{\bf Left:}  Simulation of $\CT_\lambda(t)$ for $\FD = 2$, with $t = 100$ and $\lambda = 0.1$. {\bf Right:} Simulation of $\CT_\lambda(t)$ for $\FD = 2$, with $t = 20$ and $\lambda = 0.5$.  Note that by Lemma~\ref{lem:scaling}, the laws of the simulations are the same, up to scaling the tree by $5$.}
		\label{fig:tree_simulations}
\end{figure}

\subsection{Radius of the tree}\label{sec:radius}
Let $\CR_\lambda(t) = \rad \, \CT_\lambda(t) = \sup_{x \in \CT_\lambda(t)} |x|$ denote the radius of $\CT_\lambda(t)$.  It is natural to consider the asymptotics of $\CR_\lambda(t)$ as $t \to \infty$ and by virtue of Lemma~\ref{lem:scaling} (which implies that $\CR_\lambda(t)/t$ and $\CR_{\lambda t}(1)$ have the same law), we may instead study $\CR_\lambda(1)$ as $\lambda \to \infty$. We begin by noting that the law of $\CR_\lambda(t)$ is continuous in $\lambda$.
\begin{lemma}
For each $t,r>0$, the function $\lambda \mapsto \p(\CR_\lambda(t) \leq r)$ is continuous for all $t,r>0$.
\end{lemma}
\begin{proof}
By Lemma~\ref{lem:scaling}, $\CR_\lambda(t)$ has the same distribution as $\CR_1(\lambda t)/\lambda$, which is almost surely a continuous function of $\lambda \in (0,\infty)$ for each fixed $t > 0$. To see that $\CR_\lambda(t)$ is continuous at $\lambda = 0$, note that $\p( \text{no branching occurs before time} \ t) = \p(\Exp(\lambda) > t) = e^{-\lambda t} \to 1$ as $\lambda \to 0$.
\end{proof}

Next we note the following immediate consequence of Lemma~\ref{lem:scaling}, namely that any limit of $\CR_\lambda(t)/t$ is independent of $\lambda$. Note, that at this point, it is not clear that this means anything for the radius, since $\CR_\lambda(t)$ could be of size $o(t)$. However, we shall see below that $\CR_\lambda(t)$ is actually of order $t$ as $t \to \infty$ (as opposed to a typical branch, which has radius of order $o(t)$).

\begin{corollary}\label{cor:limits_independent_of_intensity}
The laws of $\liminf_{t \to \infty} \CR_\lambda(t)/t$ and $\limsup_{t \to \infty} \CR_\lambda(t)/t$ do not depend on $\lambda$.
\end{corollary}
\begin{proof}
Follows immediately from Lemma~\ref{lem:scaling}.
\end{proof}

Now we shall lower bound the radius of the tree. Seeing as a branch will converge to the point $\origin$ as $\lambda \to \infty$ (by Lemma~\ref{lem:branch_to_zero}), it might seem reasonable to expect $\CR_\lambda(1)$ to decrease to $0$ as $\lambda \to \infty$ as well. However, we shall see that the sheer number of branches forces the radius of the tree to be positive.

\begin{proposition}\label{prop:radius_bound}
Fix $\alpha > 0$. For each $\epsilon >0$, there exists some $\lambda^* = \lambda^*(\alpha,\epsilon)>0$, such that for all $\lambda \geq \lambda^*$, we have
\begin{align*}
	\p( \CR_\lambda(1) \leq 1/2 - \epsilon) \leq (1+\alpha) e^{-\epsilon \lambda}.
\end{align*}
\end{proposition}
\begin{proof}
Note that by Proposition~\ref{prop:number_of_leaves} and Lemma~\ref{lem:asymptotics_Fs} (with $y = e^{-\lambda(1/2+\epsilon)}$, $a = \tfrac{1}{1+2\epsilon} = 1 - \tfrac{\epsilon}{1/2+\epsilon}$ and $b = 1$),
\begin{align*}
	\p( \CN_\lambda(1/2+\epsilon) \leq e^{\lambda/2}) &\leq  \sum_{n=1}^\infty \frac{\exp(-(n-1+\tfrac{\epsilon}{1/2+\epsilon})\lambda(1/2+\epsilon))}{n} \\
	&= e^{-\epsilon \lambda} \left(1 + \sum_{n=1}^\infty \frac{\exp(-n\lambda(1/2+\epsilon))}{n+1} \right) \! .
\end{align*}
Note that the sum $C_{\lambda,\epsilon} \coloneqq \sum_{n=1}^\infty \tfrac{\exp(-n\lambda(1/2+\epsilon))}{n+1}$ converges to $0$ as $\lambda \to \infty$. Thus, with probability at least $1-e^{-\epsilon \lambda} (1 + \sum_{n=1}^\infty e^{-n\lambda(1/2+\epsilon)}/(n+1) )$, there are at least $e^{\lambda/2}$ leaves at time $1/2+\epsilon$. Moreover, if the tree branches at the point $x \in B(0,1/2+\epsilon)$, then if the direction of growth for any of the two new branches belongs to the set $\CS_\half^{\FD-1}(x) = \{y \in \CS^{\FD-1}: x\cdot y \geq 0\}$, that is, the half of $\CS^{\FD-1}$ that is the closest to $x/|x|$ (this has probability $1/2$ for each branch), and the branch grows for time at least $1/2-\epsilon$ from $x$ before branching, then $\CR_\lambda(1) \geq 1/2-\epsilon$. We say that a branch grows in the \emph{good} direction if its direction of growth belongs to $\CS_\half^{\FD-1}(x)$, where $x$ denotes the most recent point of branching for the branch. Assuming that we are on the event that $\CN_\lambda(1/2+\epsilon) \geq e^{\lambda/2}$, we have by Hoeffding's bound, that the conditional probability that fewer than $\tfrac{1}{4} e^{\lambda/2}$ of the (at least) $e^{\lambda/2}$ active ends are growing in the the good direction is upper bounded by $\exp(-e^{\lambda/2}/8)$. On the event that at least $\tfrac{1}{4} e^{\lambda/2}$ leaves at time $1/2+\epsilon$ are growing in the good direction, we have by the memoryless property of the exponential distribution and Lemma~\ref{lem:asymptotics_exponential_maximum}, that the conditional probability that none of the $\tfrac{1}{4} e^{\lambda/2}$ leaves produces an interval of length at least $1/2-\epsilon$, is at most $\wt{C}_{\lambda,\epsilon} \exp(-\tfrac{1}{4} e^{\epsilon\lambda})$, where
\begin{align*}
	\wt{C}_{\lambda,\epsilon} = \exp\!\left( -\frac{1}{4} e^{\lambda/2} \sum_{n=2}^\infty \frac{e^{-n(1/2 + \epsilon)\lambda}}{n} \right) / \left(1-e^{-(1/2 + \epsilon)\lambda} \right)\! .
\end{align*}
Note that for each fixed $\epsilon>0$, $\wt{C}_{\lambda,\epsilon} \to 1$ as $\lambda \to \infty$. Hence, we have that
\begin{align*}
	\p( \CR_\lambda(1) \geq 1/2 - \epsilon) &\geq (1 - e^{-\epsilon \lambda}(1+ C_{\lambda,\epsilon})) \times (1 - \exp(-e^{\lambda/2}/8)) \times \left(1-\wt{C}_{\lambda,\epsilon} \exp\!\left( -\frac{1}{4} e^{\epsilon \lambda} \right) \right) \! .
\end{align*}
Since for large enough $\lambda$ (depending on $\alpha$ and $\epsilon$), this is lower bounded by $1 - (1+\alpha) e^{-\epsilon \lambda}$, the result follows.
\end{proof}

It turns out that a bound as in Proposition~\ref{prop:radius_bound} is the key to obtaining the main results in Theorems~\ref{thm:main_result} and~\ref{thm:main_result_1}. However, other than the calculations working nicely, there is no obvious reason why the value $1/2$ should be special or optimal for the model. Below, we shall introduce a condition which is a priori weaker than the exponential decay of the probability $\p(\CR_\lambda(1) \leq d - \epsilon)$ and might be easier to check, which turns out to imply bounds which are strong enough to give the main results of the paper with $d$ in place of $1/2$.

By virtue of Proposition~\ref{prop:radius_bound}, we now know that $\CR_\lambda(1)$ will be of constant order for all $\lambda$ large, in fact, it is extremely unlikely that it is smaller than $1/2 - \epsilon$ for large $\lambda$.  Looking at small $\lambda$, we see that increasing $\lambda$ slightly actually increases the probability $\p(\CR_\lambda(1) \leq 1/2 - \epsilon)$, since $\CR_0(1) = 1$ almost surely, and increasing the intensity increases the rate of branching, which gives $\CT_\lambda(1)$ the chance to branch and possibly keep its branches inside $B(0,1/2-\epsilon)$. However, as the intensity and hence the average number of branches increases, it eventually gets less likely that $\CT_\lambda(1)$ stays contained in $B(0,1/2-\epsilon)$ (despite the fact that the average distance from the origin travelled by a branch decreases to zero, by Lemma~\ref{lem:branch_to_zero}). We therefore find it believable that for each fixed $d \in [0,1]$, $\p(\CR_\lambda(1) \leq d)$ should either be lower bounded away from zero for all $\lambda$ large enough or converge to zero as $\lambda \to \infty$. Further evidence of this is discussed in Section~\ref{sec:connection_infinity}.

With the above discussion in mind, it is reasonable to expect that there is some maximal number $d \in [1/2,1]$ such that there exists some $p > 0$ for which $\p(\CR_\lambda(1) \geq d) \geq p$ for all $\lambda$ large enough. (Indeed, by Proposition~\ref{prop:radius_bound}, this holds for all $d < 1/2$.) We now prove a $0$-$1$-law type result, which states that the bound $\p(\CR_\lambda(1) \geq d) \geq p$ holding for all $\lambda$ large enough implies a bound which is similar to that of Proposition~\ref{prop:radius_bound} and is hence strong enough to imply the main results with $d$ in place of $1/2$.

\begin{lemma}\label{lem:zero_one_law_radius}
Fix $d \in [0,1]$. Suppose that there exist $p > 0$ and $\lambda_0 > 0$ such that $\p( \CR_\lambda(1) \geq d) \geq p$ for all $\lambda \geq \lambda_0$. Then, for each $\alpha > 0$, $\beta \in (0,1)$ and $\epsilon > 0$, there exists $\lambda_* = \lambda_*(\alpha,\beta,\epsilon,\lambda_0,p) > 0$ such that for each $\lambda \geq \lambda_*$, we have
\begin{align*}
	\p( \CR_\lambda(1) \leq d - \epsilon) \leq (1+\alpha) e^{-\epsilon \beta\lambda/2}.
\end{align*}
\end{lemma}
For a bound that holds for all $\lambda > 0$, $\beta \in (0,1)$ and $\epsilon > 0$, see~\eqref{eq:lower_bound_radius_exponential}.
\begin{proof}
By the assumption of the lemma and Lemma~\ref{lem:scaling}, we have that if $\epsilon > 0$, then
\begin{align}\label{eq:lower_bound_radius_d}
	\p( \CR_\lambda(1-\epsilon/2) \geq d-\epsilon/2) \geq p,
\end{align}
for all $\lambda \geq \lambda_0/(1-\epsilon/2)$ and that if any descending tree of a point in $\CT_\lambda(\epsilon/2)$ has radius at least $d - \epsilon/2$ within time $1-\epsilon/2$ of its starting time, then $\CR_\lambda(1) \geq d - \epsilon$. Moreover, for fixed $\beta \in (0,1)$, we note that by Lemma~\ref{lem:asymptotics_Fs} (with $y = e^{-\epsilon \lambda/2}$, $a = 1-\beta$ and $b = 1$),
\begin{align}\label{eq:lower_bound_leaves_delta_half}
	\p( \CN_\lambda(\epsilon/2) \geq e^{(1-\beta) \epsilon \lambda /2}) \geq 1 - e^{-\epsilon \beta \lambda /2} \! \left( 1 + \sum_{n=1}^\infty \frac{e^{-n \epsilon \lambda/2}}{n+1} \right)\! .
\end{align}
By~\eqref{eq:lower_bound_radius_d}, we have the following lower bound on the conditional probability that that $\CR_\lambda(1) \geq d - \epsilon$ given the event $\{ \CN_\lambda(\epsilon/2) \geq e^{(1-\beta)\epsilon \lambda/2} \}$,
\begin{align}\label{eq:lower_bound_radius_leaves}
	\p( \CR_\lambda(1) \geq d-\epsilon \, | \,  \CN_\lambda(\epsilon/2) \geq e^{(1-\beta)\epsilon\lambda/2}) &\geq 1 - (1-p)^{\exp((1-\beta)\epsilon \lambda/2)} \nonumber \\
	&= 1 - \exp\!\left( e^{(1-\beta)\epsilon \lambda/2} \log(1-p) \right) \! .
\end{align}
Consequently, by~\eqref{eq:lower_bound_leaves_delta_half} and~\eqref{eq:lower_bound_radius_leaves}, we have that
\begin{align}\label{eq:lower_bound_radius_exponential}
	\p( \CR_\lambda(1) \geq d - \epsilon) \geq \left( 1 - \exp\!\left( e^{(1-\beta)\epsilon \lambda/2} \log(1-p) \right) \right) \! \left( 1 - e^{-\epsilon \beta \lambda/2} \! \left( 1 + \sum_{n=1}^\infty \frac{e^{-n \epsilon\lambda /2}}{n+1} \right) \right) \! .
\end{align}
Since $\sum_{n=1}^\infty e^{-{n \epsilon \lambda /2}}/(n+1)$ converges to $0$ as $\lambda \to \infty$ and $ \exp( e^{(1-\beta)\epsilon \lambda/2} \log(1-p) ) = o(e^{-\epsilon\beta \lambda /2})$ as $\lambda \to \infty$, we have that there exists $\lambda_* = \lambda_*(\alpha,\beta,\epsilon,\lambda_0,p) \geq \lambda_0/(1-\epsilon/2)$ such that $\p( \CR_\lambda(1) \leq d - \epsilon) \leq (1+\alpha) e^{-\epsilon \beta \lambda /2}$ for all $\lambda \geq \lambda_*$.
\end{proof}

Since the condition that there exist $d,p,\lambda_0 > 0$ such that $\p(\CR_\lambda(1) \geq d) \geq p$ is satisfied for each $d < 1/2$ and by Lemma~\ref{lem:zero_one_law_radius} guarantees a bound of similar sort as that of Proposition~\ref{prop:radius_bound}, we now take this assumption as the basis of our proofs. This makes it easier to strengthen the results, should one manage to prove said condition for larger $d$.

\begin{corollary}\label{cor:liminf_lower_bound}
Suppose that $d ,p,\lambda_0>0$ are such that $\p( \CR_\lambda(1) \geq d) \geq p$ for all $\lambda \geq \lambda_0$. Then, almost surely,
\begin{align*}
	\liminf_{t \to \infty} \CR_\lambda(t)/t \geq d.
\end{align*}
\end{corollary}
\begin{proof}
For $k \in \N$ we consider the event $E_k = \{\CR_\lambda(k)/k \leq d - \epsilon \}$. We note that if $E_k$ does not occur, then
\begin{align}\label{eq:radius_ratio}
	\frac{\CR_\lambda(t)}{t} \geq \left(d - \epsilon \right) \frac{k}{k+1}, \quad k \leq t \leq k+1.
\end{align}
Let $F_k = \{ \exists j \geq k: E_j \ \text{occurs} \}$ and note that by~\eqref{eq:radius_ratio}, we have for each $k \in \N$ that
\begin{align*}
	\left\{ \liminf_{t \to \infty} \CR_\lambda(t)/t \leq (d - \epsilon) \tfrac{k}{k+1} \right\} \subset F_k.
\end{align*}
Moreover, we note that $F_k \subset \cup_{j \geq k} E_j$ and by Lemma~\ref{lem:zero_one_law_radius} with, say, $\alpha = 1$ and $\beta = 1/2$, we have for $k \geq \lambda_*(\alpha,\beta,\epsilon,\lambda_0,p)/\lambda$, that
\begin{align*}
	\p( E_k ) = \p( \CR_{k \lambda}(1) \leq d - \epsilon) \leq 2 e^{-k \epsilon \lambda/4}.
\end{align*}
Consequently, for $k \geq \lambda_*(\alpha,\beta,\epsilon)/\lambda$ we have that
\begin{align*}
	\p( F_k ) \leq \sum_{j \geq k} \p( E_j ) \leq \sum_{j \geq k} 2 e^{-j\epsilon \lambda/4} = \frac{2 e^{-k \epsilon \lambda/4}}{1 - e^{-\epsilon \lambda / 4}}.
\end{align*}
This converges to $0$ as $k \to \infty$ and hence $\p( \liminf_{t \to \infty} \CR_\lambda(t)/t \leq d - \epsilon) = 0$. Since $\epsilon > 0$ was arbitrary, the result follows.
\end{proof}

\begin{proof}[Proof of Theorem~\ref{thm:main_result}]
The first part is Proposition~\ref{prop:radius_bound} and the second part is Corollaries~\ref{cor:liminf_lower_bound} and~\ref{cor:limits_independent_of_intensity}.
\end{proof}

\subsection{Connection probabilities}\label{sec:connection}
Fix some point $x \in \R^\FD$. Let
\begin{align*}
    p_r^\lambda(t,x) \coloneqq \p( \dist(\CT_\lambda(t),x) \leq r),
\end{align*}
and note that $p_r^\lambda(t,x)$ is radially symmetric in $x$, that is, its dependence on $x$ is only on $|x|$. Thus, we may and will sometimes write $p_r^\lambda(t,d)$ when $|x|=d$. Furthermore, we let
\begin{align*}
    q_r^\lambda(t,x) = \p( \dist(\CT_\lambda(t),x) > r) = 1- p_r^\lambda(t,x).
\end{align*}
Assume for now that $\FD = 2$. Then, we have from Section~3.2 of \cite{antz2015network}, the following functional equation for $q_r^\lambda(t,x)$. For $|x|=d$, $t>d-r$ and $d>r$,
\begin{align}\label{eq:connection_equation}
    q_r^\lambda(t,d) = &\int_{-\alpha_0}^{\alpha_0} \int_0^{s_0(\alpha)} \frac{\lambda e^{-\lambda s}}{2\pi} q_r^\lambda(t-s,D(s,\alpha))^2 ds d\alpha \nonumber \\
    &+ \int_{\alpha_0}^{2\pi-\alpha_0} \int_0^t \frac{\lambda e^{-\lambda s}}{2 \pi} q_r^\lambda(t-s,D(s,\alpha))^2 ds d\alpha + e^{-\lambda t} f_r^\lambda(t,d),
\end{align}
where $\alpha_0 = \arcsin\tfrac{r}{d}$, $s_0(\alpha) = d \cos\alpha - \sqrt{r^2-d^2\sin^2\alpha}$, $D(s,\alpha) = \sqrt{d^2+s^2-2ds\cos\alpha}$, and
\begin{align*}
    f_r^\lambda(t,d) &= \p(\dist(\CT_\lambda(t),x) > r \,| \, \textup{no branching has occurred before time} \, t) \\
    &= \begin{cases}
                1-\frac{1}{\pi} \arccos\!\left( \frac{d^2+t^2-r^2}{2td} \right)\! , &\textup{if } d-r \leq t \leq \sqrt{d^2-r^2}, \\
                1-\frac{1}{\pi} \arcsin \frac{r}{d}, &\textup{if } t> \sqrt{d^2-r^2},
                \end{cases}
\end{align*}
with boundary conditions
\begin{align*}
    &q_r^\lambda(t,d) = 0, \quad \textup{if } d \leq r, \ t \geq 0, \\
    &q_r^\lambda(t,d) = 1, \quad \textup{if } d>r, \ 0 \leq t <d-r.
\end{align*}
While this does not yield an analytic solution, it does allow for numerical computation. Similar equations can be derived in any dimension. For the case $\FD = 3$, see Section~3.2 of \cite{gt2019decay}.

Now, allowing for general $\FD$, we have (see equation~(12) of \cite{gt2019decay}) that
\begin{align*}
	p_r^\lambda(t,x) \geq e^{-\lambda t} p_r^0(t,x).
\end{align*}
Moreover, it is obvious from the definition that if $r_1 \leq r_2$ and $t_1 \leq t_2$, then
\begin{align}\label{eq:monotonicity}
	p_{r_1}^\lambda(t_1,x) \leq p_{r_2}^\lambda(t_2,x).
\end{align}
Using Lemma~\ref{lem:scaling}, we have the following identities. Let $s > 0$ and $d = |x|$. Then,
\begin{align}\label{eq:scaling_probabilities}
	p_r^\lambda(t,x) = p_{r s/t}^{\lambda t/s}(s,x s/t) = p_{r/t}^{\lambda t}(1,x/t) = p_{r/d}^{\lambda d}(t/d,1).
\end{align}
Let $a > 1$ and note that by~\eqref{eq:monotonicity} and~\eqref{eq:scaling_probabilities},
\begin{align*}
	p_r^\lambda(t,x) \leq p_r^{a\lambda}(t,x/a).
\end{align*}
We expect that in the variables $\lambda$ and $x$, the functions $p_r^\lambda(t,x)$ do not display the same monotonicity. Indeed, one can expect that for $x$ close to $0$,  $\lambda_1 \leq \lambda_2$ implies that $p_r^{\lambda_1}(t,x) \leq p_r^{\lambda_2}(t,x)$ (since if the first branch misses the ball $B(x,r)$, then a larger branching intensity gives the tree more ``attempts'' at hitting it). On the other hand, we know that the conditional probability that $\CR_\lambda(t) = t$ given that a branching occurs is zero, so the maximal radius is only attained if the tree never branches, which is more likely to occur if $\lambda$ is smaller. Hence, it is reasonable to expect that $p_r^{\lambda_1}(t,x) \geq p_r^{\lambda_2}(t,x)$ for $x$ close to $\partial B(0,t)$ and $r$ small, whenever $\lambda_1$ and $\lambda_2$ are small. But for large $\lambda_1$ and $\lambda_2$, it seems reasonable to expect that the increase in the number of branches is what will impact the connection probability the most. In other words, for $x$ close to $\partial B(0,t)$, we expect that $\lambda \mapsto p_r^\lambda(t,x)$ is decreasing on some interval $[0,\lambda_0]$ and then increasing on $[\lambda_0,\infty)$. 

\begin{remark}\label{rmk:other_functional_equations}
For $r>0$, $\lambda > 0$, $x \in \R^\FD$ and $t > 0$, let $\CN_r^\lambda(t,x) = |\CS_\lambda(t) \cap B(x,r)|$, that is, $\CN_r^\lambda(t,x)$ is the number of branching points of $\CT_\lambda(t)$ that lie in $B(x,r)$, and let $\SN_r^\lambda(t,x) = \E[ \CN_r^\lambda(t,x)]$. Then, similarly to the above, one can derive a functional equation for $\SN_r^\lambda(t,x)$, see~\cite{goriachkin2023thesis}. For $D \subset \R^\FD$, we let $\CL_\lambda(t,D)$ be the total length of the intervals of $\CT_\lambda(t) \cap D$. If we write $\ell_r^\lambda(t,x) = \E[ \CL_\lambda(t,B(x,r))]$, then in the same way one can derive a functional equation for it. Moreover, similarly to the equations for $q_r^\lambda(t,x)$, the equations for $\SN_r^\lambda(t,x)$ and $\ell_r^\lambda(t,x)$ do not give us analytic solutions, but could be used for numerical computations.
\end{remark}

\subsubsection{Connection probabilities as $\lambda \to \infty$}\label{sec:connection_infinity}
Again, let us first focus on the case $\FD = 2$. We would like to consider $p_r^\infty(t,d) = \lim_{\lambda \to \infty} p_r^\lambda(t,d)$, provided that the limit exists. Equivalently, we can consider $q_r^\infty(t,d) = \lim_{\lambda \to \infty} q_r^\lambda(t,d)$ (provided that the limit exists). For now, assume that the limit exists (in some sufficiently good sense).

We note that if $\Ff_\lambda(s) = \lambda e^{-\lambda s}$, then $\Ff_\lambda \to \delta_0$ in the space of tempered distributions on the Schwartz space as $\lambda \to \infty$. In a space of distributions over a slightly more general function space (allowing for discontinuities), $\Ff_\lambda \to \delta_{0+}$ where $\delta_{0+}$ is the distribution acting on functions as $\delta_{0+}[\varphi] = \lim_{x \searrow 0} \varphi(x)$. Thus, taking the limit as $\lambda \to \infty$ in~\eqref{eq:connection_equation} it is reasonable to expect that
\begin{align*}
	q_r^\infty(t,d) = q_r^\infty(t,d)^2,
\end{align*}
that is, $q_r^\infty(t,d) \in \{0,1\}$, whenever $(t,d)$ is a continuity point of $q_r^\infty$. The same analysis can be carried out in the case $\FD = 3$. Next we prove that there is at least an interval of $d$ where it is true for arbitrary dimension $\FD$ (and we expect the statement to be true for all $d \in [0,1]$).

\begin{proposition}\label{prop:connection_probability_infinity}
Suppose that $d ,p,\lambda_0>0$ are such that $\p( \CR_\lambda(1) \geq d) \geq p$ for all $\lambda \geq \lambda_0$. Fix $x \in B(0,d)$ and let $r > 0$ be such that $|x| + r < d$ and $\epsilon = \min(d-|x|,1/100)$. Then, for each $a \in (0,1)$ and $\alpha > 0$, there exists $\lambda_\diamond = \lambda_\diamond(\alpha,a,\epsilon,\lambda_0,p,r)$ such that
\begin{align*}
p_r^\lambda(1,x) \geq 1 - (1+\alpha)e^{-(1-a)\lambda \epsilon /4},
\end{align*}
for all $\lambda \geq \lambda_\diamond$.
\end{proposition}
\begin{remark}
Proposition~\ref{prop:connection_probability_infinity} essentially states that as $\lambda \to \infty$, the entire ball $B(0,d)$ will be filled by $\CT_\lambda(1)$ and hence that if $B(x,r) \cap B(0,d) \neq \emptyset$, then $\CT_\lambda(1) \cap B(x,r) \neq \emptyset$. 
\end{remark}
\begin{proof}[Proof of Proposition~\ref{prop:connection_probability_infinity}]
Consider $x$, $r$ and $\epsilon$ as in the statement of the proposition. For a fixed $\beta \in (0,1)$ we have by~\eqref{eq:lower_bound_radius_exponential}, symmetry and Lemma~\ref{lem:scaling}, that for all $y \in B(0,d-\epsilon/2)$ with $|y| > r$,
\begin{align*}
	&\p( \dist(\CT_\lambda(1-\epsilon/4),y) \leq r) \\
	&\geq \frac{c_{r/|y|}^{\FD-1}}{c^{\FD-1}} \p(\CR_\lambda(1-\epsilon/4) \geq |y|) \\
	&\geq \frac{c_{r/|y|}^{\FD-1}}{c^{\FD-1}} \p(\CR_\lambda(1-\epsilon/4) \geq d-\epsilon/2) \\
	&\geq \frac{c_{r/|y|}^{\FD-1}}{c^{\FD-1}} \p\!\left(\CR_{\lambda(1-\epsilon/4)}(1) \geq d-\frac{\epsilon/4}{1-\epsilon/4} \right) \\
	&\geq \frac{c_{r/|y|}^{\FD-1}}{c^{\FD-1}} \left( 1 - \exp\!\left( e^{(1-\beta)\epsilon \lambda/8} \log(1-p) \right) \right) \! \left( 1 - e^{-\epsilon \beta \lambda/8} \! \left( 1 + \sum_{n=1}^\infty \frac{e^{-n \epsilon\lambda /8}}{n+1} \right) \right) \! ,
\end{align*}
where for $s > 0$, $c_s^{\FD-1}$ is the $(\FD-1)$-dimensional surface area of $\CS^{\FD-1} \cap B((1,0,\dots,0),s)$ and $c^{\FD-1}$ is the $(\FD-1)$-dimensional surface area of $\CS^{\FD-1}$. We remark that there exists a universal constant $C \geq 1$ such that $C^{-1} s^{\FD-1} \leq c_s^{\FD-1}/c^{\FD-1} \leq C s^{\FD-1}$ for all $s \in (0,1)$ and that, obviously, $c_s^{\FD-1}/c^{\FD-1} \leq 1$. 
By the same argument as in the last paragraph of Lemma~\ref{lem:zero_one_law_radius}, there exist $\lambda_\dagger = \lambda_\dagger(\beta,\epsilon,\lambda_0,p) > 0$ such that for $\lambda \geq \lambda_\dagger$,
\begin{align}
	\p( \dist(\CT_\lambda(1-\epsilon/4),y) \leq r) \geq \frac{c_{r/|y|}^{\FD-1}}{2 c^{\FD-1}} \geq \frac{1}{2C} \frac{r^{\FD-1}}{|y|^{\FD - 1}} \geq \frac{r^{\FD-1}}{2 C}.
\end{align}
In the case $|y| \leq r$, we obviously have that $\p( \dist(\CT_\lambda(1-\epsilon/4),y) \leq r) = 1$. Since for any $y \in \ol{B(0,\epsilon/4)}$, we have that $\dist(x,y) \leq d - 3\epsilon/4$, arguing as in the proof of Lemma~\ref{lem:zero_one_law_radius}, we have for $\lambda \geq \lambda_\dagger$, that
\begin{align*}
	p_r^\lambda(1,x) &\geq \p( \dist(\CT_\lambda(1),x) \leq r \, | \, \CN_\lambda(\epsilon/4) \geq e^{\lambda \epsilon a/4}) \p( \CN_\lambda(\epsilon/4) \geq e^{\lambda \epsilon a/4} ) \\
	&\geq \left(1 - \left( 1- \frac{r^{\FD-1}}{2 C} \right)^{\exp(\lambda \epsilon a /4)} \right) \! \left(1 - e^{-(1-a) \lambda\epsilon/4} \left( 1 + \sum_{n=1}^\infty \frac{e^{-n \lambda \epsilon/4}}{n+1} \right) \right) \\
	&= \left(1 - \exp\!\left(\frac{\lambda \epsilon a}{4} \log \left( 1- \frac{r^{\FD-1}}{2 C} \right)\right) \right)  \! \left(1 - e^{-(1-a) \lambda\epsilon/4} \left( 1 + \sum_{n=1}^\infty \frac{e^{-n \lambda \epsilon/4}}{n+1} \right) \right) \! .
\end{align*}
Similarly to the above, since $\exp(\tfrac{1}{4}\lambda \epsilon a \log ( 1- r^{\FD-1}/(2 C) )) = o(e^{-(1-a)\lambda\epsilon/4})$ and $\sum_{n=1}^\infty e^{-n \lambda \epsilon/4}/(n+1)$ converges to $0$ as $\lambda \to \infty$, it follows that for any $\alpha > 0$, there exists $\lambda_\diamond = \lambda_\diamond(\alpha,a,\beta,\epsilon,\lambda_0,p,r) \geq \lambda_\dagger$ such that
\begin{align*}
	p_r^\lambda(1,x) \geq 1 - (1 + \alpha) e^{-(1-a)\lambda \epsilon/4}.
\end{align*}
Finally, we note that we may drop the dependence of $\lambda_\diamond$ on $\beta$, since $\beta$ was just an arbitrary constant in $(0,1)$ chosen in this proof.
\end{proof}

The statement of Proposition~\ref{prop:connection_probability_infinity} essentially says that there is a ball centred at the origin which gets filled by the tree $\CT_\lambda(1)$ as $\lambda \to \infty$. When translating this to the case of $\CT_\lambda(t)$ via scaling (Lemma~\ref{lem:scaling}) this would then tell us that, while the process $\CT_\lambda(t)$ may not have the same sharp behaviour as $t \to \infty$, it would roughly be a disk of radius $t d$, where each point is at most distance $o(t)$ from $\CT_\lambda(t)$ for large $t$. 

\begin{proposition}\label{prop:convergence_time_1}
Suppose that $d ,p,\lambda_0>0$ are such that $\p( \CR_\lambda(1) \geq d) \geq p$ for all $\lambda \geq \lambda_0$. Then $\CT_\lambda(1) \cap B(0,d) \to B(0,d)$ in probability with respect to the topology determined by the Hausdorff distance as $\lambda \to \infty$.
\end{proposition}
\begin{proof}
Let $X_\delta = (\tfrac{\delta}{2} \Z^\FD) \cap B(0,d - \delta/4)$ and note that if $\dist( \CT_\lambda(1) \cap B(0,d), x) \leq \delta$ for each $x \in X_\delta$, then $\dist_H(\CT_\lambda(1) \cap B(0,d),B(0,d)) \leq \delta$, where $\dist_H$ denotes the Hausdorff distance. Note that the number of points in $X_\delta$ is upper bounded by $8 d^2 \delta^{-2\FD}$ (this bound is not sharp). Thus, fixing some $\alpha > 0$ and $a \in (0,1)$, we have by Proposition~\ref{prop:connection_probability_infinity}, that for each $\delta < 1/25$ and all $\lambda \geq \lambda_\diamond(\alpha,a,\delta/4,\lambda_0,p,\delta)$ (where $\lambda_\diamond$ is as in Proposition~\ref{prop:connection_probability_infinity}),
\begin{align}\label{eq:hausdorff_distance_decay}
	\p( \dist_H(\CT_\lambda(1) \cap B(0,d),B(0,d)) > \delta) &\leq \p( \exists x \in X_\delta: \dist(\CT_\lambda(1),x) > \delta/2) \nonumber \\
	&\leq \sum_{x \in X_\delta} \p( \dist(\CT_\lambda(1),x) > \delta) \nonumber \\
	&\leq 8 d^2 \delta^{-2\FD} (1+\alpha) e^{-(1-a)\lambda \delta/16}.
\end{align}
Thus, as $\lambda \to \infty$,~\eqref{eq:hausdorff_distance_decay} decays to $0$. Since $\delta > 0$ was arbitrary, the proof is done. 
\end{proof}

We record the following estimate that follows from the above proof.

\begin{lemma}\label{lem:holes_in_tree_time_t}
Suppose that $d ,p,\lambda_0>0$ are such that $\p( \CR_\lambda(1) \geq d) \geq p$ for all $\lambda \geq \lambda_0$. Fix $C > 8 d^2$ and $a \in (0,1)$. Then, there exists $t_* = t_*(\lambda,C,a,\delta,\lambda_0,p)$ such that
\begin{align*}
	\p( \exists x \in B(0,t d): \dist(\CT_\lambda(t),x) > \delta t) \leq C \delta^{-2\FD} e^{-\tfrac{(1-a)}{16} \lambda \delta t},
\end{align*}
for all $t \geq t_*$.
\end{lemma}
\begin{proof}
By Lemma~\ref{lem:scaling} and~\eqref{eq:hausdorff_distance_decay}, letting $X_\delta$ be as in the proof of Proposition~\ref{prop:connection_probability_infinity}, and fixing $\alpha > 0$ and $a \in (0,1)$, we have that
\begin{align*}
	\p( \exists x \in B(0,t d): \dist(\CT_\lambda(t),x) > \delta t) &= \p( \exists x \in B(0,d): \dist(\CT_\lambda(t)/t,x) > \delta) \\
	&= \p( \exists x \in B(0,d): \dist(\CT_{\lambda t}(1),x) > \delta) \\
	&\leq \sum_{x \in X_\delta} \p( \dist(\CT_{\lambda t}(1),x) > \delta) \\
	&\leq 8 d^2 \delta^{-2\FD} (1+\alpha) e^{-\tfrac{(1-a)}{16} \lambda \delta t},
\end{align*}
whenever $\lambda t \geq \lambda_\diamond(\alpha,a,\delta/4,\lambda_0,p,\delta)$ and $\lambda_\diamond$ is as in Proposition~\ref{prop:connection_probability_infinity}. That is, if we pick $\alpha > 0$ such that $C = 8(1+\alpha) d^2$, then for all $t \geq t_*(\lambda,C,a,\delta,\lambda_0,p) = \lambda_\diamond(\alpha,a,\delta/4,\lambda_0,p,\delta)/\lambda$, we have that
\begin{align*}
	\p( \exists x \in B(0,t d): \dist(\CT_\lambda(t),x) > \delta t) \leq C \delta^{-2\FD} e^{-\tfrac{(1-a)}{16} \lambda \delta t}.
\end{align*}
Thus, the proof is complete.
\end{proof}

\begin{remark}
In the functional equations mentioned in Remark~\ref{rmk:other_functional_equations}, the same procedure as above (letting $\lambda \to \infty$), does not indicate any nontrivial behaviour. Indeed, it is obvious that the number of branching points and total length in a ball will be either $0$ (if the tree does not intersect it) or infinite (if the tree hits the ball, then as $\lambda \to \infty$, the number of branching points occurring in said ball will grow exponentially with $\lambda$ and hence the number and total length of intervals as well. This is reflected, for example, in the functional equation essentially turning into the equation $\ell_r^\infty(t,d) = 2 \ell_r^\infty(t,d)$ as $\lambda \to \infty$, and hence that $\ell_r^\infty(t,d) \in \{0,\infty\}$. That is, the expected total length of all the intervals of $\CT_\lambda(t)$ inside $B(d,r)$ as $\lambda \to \infty$ would be either zero or infinite.
\end{remark}

\begin{proof}[Proof of Theorem~\ref{thm:main_result_1}]
The first part is Proposition~\ref{prop:convergence_time_1} and the last part follows by combining Lemma~\ref{lem:holes_in_tree_time_t} with Proposition~\ref{prop:radius_bound}.
\end{proof}

\appendix

\section{Asymptotics of first success and exponential distributions}\label{app:asymptotics}
In this appendix we prove the necessary asymptotics for first success and exponential distributions, that are used to prove the radius bounds in Section~\ref{sec:radius}. For each $y \in (0,1)$, let $N_y \sim \Fs(y)$.
\begin{lemma}\label{lem:asymptotics_Fs}
We have that $yN_y \to \Exp(1)$ in distribution as $y \to 0$. Moreover,
if $0 < a < 1$ and $b>0$, then
\begin{align}
	\p( N_y \leq b y^{-a} ) \leq b \sum_{n=1}^\infty \frac{y^{n-a}}{n}.
\end{align}
\end{lemma}
\begin{proof}
Fix $0 < a \leq 1$ and $b > 0$, and recall that $\log(1-y) = - \sum_{n=1}^\infty y^n /n$ whenever $|y| < 1$. We have that
\begin{align*}
	\p( N_y \leq b y^{-a}) &= 1 - (1-y)^{b/y^a} (1-y)^{\lfloor b/y^a \rfloor - b/y^a} \\
	&= 1 - e^{b y^{-a} \log(1-y)} (1-y)^{\lfloor b/y^a \rfloor - b/y^a} \\
	&= 1 - \exp\!\left(-b \sum_{n=1}^\infty \frac{y^{n-a}}{n} \right) (1-y)^{\lfloor b/y^a \rfloor - b/y^a}.
\end{align*}
Since $1 \leq (1-y)^t \leq 1/(1-y)$ whenever $-1 \leq t \leq 0$, we have that
\begin{align*}
	1 - \frac{1}{1-y} \exp\!\left(-b \sum_{n=1}^\infty \frac{y^{n-1}}{n} \right) \leq \p( yN_y \leq b) \leq 1 - e^{-b},
\end{align*}
and thus $y N_y \to \Exp(1)$ in distribution. Similarly, using that $e^{-t} > 1-t$ for all $t>0$, we have for $0 < a < 1$ that
\begin{align*}
	\p( N_y \leq b y^{-a}) \leq b \sum_{n=1}^\infty \frac{y^{n-a}}{n}.
\end{align*}
\end{proof}

For each $x > 0$, let $(T_x^j)_{j \in \N^+}$ be an i.i.d.\ collection of $\Exp(x)$-distributed random variables.

\begin{lemma}\label{lem:asymptotics_exponential_maximum}
Let $a > c > 0$ and $b > 0$. Then, for all $x > 0$,
\begin{align*}
	\p\!\left( \max_{j \leq be^{ax}} \, T_x^j \leq c \right) \leq C_{a,b,c,x} \exp( - b e^{(a-c)x}),
\end{align*}
where $C_{a,b,c,x} = \exp(-b e^{ax} \sum_{n=2}^\infty \tfrac{e^{-ncx}}{n})/(1-e^{-cx})$.
\end{lemma}
\begin{proof}
We have that
\begin{align*}
	\p\!\left( \max_{j \leq b e^{ax}} \, T_x^j \leq c \right) &= \p( T_x^1 \leq c )^{\lfloor b e^{ax} \rfloor} = (1 - e^{-cx})^{b e^{ax}} (1 - e^{-cx})^{\lfloor b e^{ax} \rfloor - b e^{ax}} \\
	&= \exp( b e^{ax} \log(1-e^{-cx})) (1 - e^{-cx})^{\lfloor b e^{ax} \rfloor - b e^{ax}} \\
	&= \exp\!\left( - b e^{ax} \sum_{n=1}^\infty \frac{e^{-n c x}}{n} \right) (1 - e^{-cx})^{\lfloor b e^{ax} \rfloor - b e^{ax}} \\
	&\leq \exp\!\left( - b e^{ax} \sum_{n=1}^\infty \frac{e^{-n c x}}{n} \right) /(1-e^{-cx}) \\
	&\leq C_{a,b,c,x} \exp( - be^{(a-c)x}),
\end{align*}
where, in the second to last inequality, we again used the inequality $(1-y)^t \leq 1/(1-y)$ which holds for $t \in [-1,0]$.
\end{proof}

\bibliographystyle{abbrv}
\bibliography{bibliography}

\end{document}